\documentclass[12pt,reqno]{amsart}

\usepackage{amssymb}
\usepackage[all]{xy}

\usepackage{hyperref}

\numberwithin{equation}{section}

\newtheorem{theorem}{Theorem}[section]

\theoremstyle{definition}
\newtheorem{definition}[theorem]{Definition}
\newtheorem{remark}[theorem]{Remark}

\begin{document}

\baselineskip=15pt

\title[Holomorphic Lie algebroid connections in higher dimensions]{Existence of holomorphic Lie algebroid 
connections in higher dimensions}

\author[I. Biswas]{Indranil Biswas}

\address{Department of Mathematics, Shiv Nadar University, NH91, Tehsil Dadri,
Greater Noida, Uttar Pradesh 201314, India}

\email{indranil.biswas@snu.edu.in, indranil29@gmail.com}

\author[A. Singh]{Anoop Singh}

\address{Department of Mathematical Sciences, Indian Institute of Technology (BHU), Varanasi 
221005, India}

\email{anoopsingh.mat@iitbhu.ac.in}



\date{}

\begin{abstract}
Let $(V,  \phi)$ be a holomorphic Lie algebroid over an irreducible smooth complex projective variety
$X$ of dimension at least three, and let $E$ be a holomorphic vector bundle on $X$. We establish a
necessary and sufficient condition for the existence of a holomorphic $(V, \phi)$--connection on $E$.
\end{abstract}

\maketitle

\section{Introduction}

Let $X$ be an irreducible smooth projective variety over the field $\mathbb{C}$ of complex numbers.
Holomorphic connections on holomorphic vector bundles over $X$ were introduced by
Atiyah \cite{At}. Holomorphic Lie algebroid connections (see Definition \ref{Def-3}) over $X$
provide a natural generalization of holomorphic connections, where
the tangent Lie algebroid $(TX,  \mathrm{Id}_{TX})$ is replaced by
an arbitrary holomorphic Lie algebroid $(V,  \phi)$ over $X$ (see Definition \ref{Def-1} for
holomorphic Lie algebroids).

It may be mentioned that by choosing the Lie algebroid appropriately, a wide range of other 
algebraic and differential geometric objects can be interpreted as Lie algebroid 
connections. For instance, Higgs bundles, \cite{Hi}, \cite{Si1}, twisted Higgs bundles, 
\cite{Ni1}, \cite{GGPN}, logarithmic connections, \cite{De}, \cite{Ni2}, meromorphic connections, 
\cite{Bo}, \cite{BS2}, and a broad subclass of Simpson's notion of $\Lambda$-modules
can all be understood as Lie algebroid connections for suitable choices of the Lie 
algebroid $(V,  \phi)$ (see, \cite{Si2}, \cite{To2}, \cite{AO}).

As far as holomorphic connections are concerned, not every holomorphic vector bundle over an
irreducible smooth complex projective variety $X$ admits a holomorphic connection. The existence of 
holomorphic connections is a highly restrictive condition. A holomorphic vector bundle over a 
compact Riemann surface does not always admit a holomorphic connection. There is a criterion 
that determines when such a connection exists. On a compact connected Riemann surface, a
holomorphic vector bundle $F$ admits a holomorphic connection if and only if the degree of each
indecomposable component of $F$ is zero \cite{At}, \cite{We}.

Further, if $X$ is a compact connected K\"ahler Calabi-Yau manifold, meaning $c_1(TX)  =  0$,
then any holomorphic vector bundle over $X$ admitting a
holomorphic connection also admits a flat holomorphic connection \cite{BD}.

See \cite{ABKS} for a criterion for the existence of Lie algebroid connections on holomorphic vector bundles 
over a compact connected Riemann surface.

Atiyah proved in \cite[Proposition~21]{At} a criterion for the existence of holomorphic connections on an
irreducible smooth projective variety $ X$ over $\mathbb{C}$ of complex dimension at least $3$. In this
work, we aim to generalize Atiyah’s criterion to the setting of Lie algebroid connections in higher
dimensions. More precisely, we establish the following result (see Theorem \ref{thm:main}).

\begin{theorem}
\label{thm:main0}
Assume that $X   \subset  {\mathbb C}{\mathbb P}^N$ is an irreducible
smooth complex projective variety of dimension at least three. Let $(V,  \phi)$ be
a holomorphic Lie algebroid on $X$ such that the anchor map $\phi   :   V 
\longrightarrow  TX$ is surjective. A holomorphic vector bundle $E$ over $X$ admits a holomorphic
$(V,  \phi)$-connection if and only if the restriction of $E$ to the intersection of $X$ with a sufficiently
positive hypersurface in ${\mathbb C}{\mathbb P}^N$ admits a holomorphic Lie algebroid connection.
\end{theorem}

In addition to the general case, we also consider Lie algebroids associated to reduced
effective divisors. For a reduced effective divisor $Z \subset X$, we consider the Lie algebroid
$$j : W = TX \otimes \mathcal{O}_X(-Z)  \hookrightarrow  TX.$$ We prove that, under the
same dimensional hypothesis, the existence of a holomorphic $(W, j)$-connection on a vector bundle $E$ over
$X$ is equivalent to the
existence of such a connection on its restriction to some smooth
hypersurface $X_n$ of sufficiently large degree $n$, which intersects $Z$ properly; see Theorem \ref{thm:red}.

\section{Lie algebroids and connections}

Let $X$ be an irreducible smooth projective variety over $\mathbb{C}$. The holomorphic tangent and
cotangent bundles of $X$ will be denoted by $TX$ and $\Omega^1_X$ respectively.

\begin{definition}\label{Def-1}
A Lie algebroid on $X$ is a holomorphic vector bundle $V \longrightarrow X$,
together with an ${\mathcal{O}}_X$--linear homomorphism $$\phi : V  \longrightarrow TX,$$ called the {\it anchor
map}, and a structure of a $\mathbb{C}$--Lie algebra on the sheaf of locally defined holomorphic sections of $V$
$$
[-, \,-]  : V \otimes_{\mathbb{C}} V  \longrightarrow V
$$
such that $$[s,\, f\cdot t]  =  f \cdot [s,\, t]+\phi(s)(f) \cdot t$$ for all locally defined
holomorphic sections $s,\, t$ of $V$ and all locally defined holomorphic functions $f$ on $X$.
\end{definition}

Note that the pair $(TX,  {\rm Id}_{TX})$ is a Lie algebroid; the Lie algebra structure on
$TX$ is given by the Lie bracket operation of vector fields.

\begin{definition}
\label{Def-2}
A Lie algebroid $(V, \phi)$ on $X$ is called \textit{split} if there is a holomorphic
${\mathcal O}_X$--linear homomorphism $\eta : TX \longrightarrow V$ such that
\begin{equation}\label{eta}
\phi\circ\eta  =  \text{Id}_{TX}.
\end{equation}
A Lie algebroid $(V, \phi)$ on $X$ is called \textit{nonsplit} if it is not split.
\end{definition}

\begin{definition}
\label{Def-3}
Take a Lie algebroid $(V, \phi)$ on $X$.
Let $\phi^*: \Omega^1_X  \longrightarrow V^*$ be the dual of $\phi$.
A \emph{Lie algebroid connection} on a holomorphic vector bundle $E$ on $X$ 
is a $\mathbb{C}$--linear holomorphic map 
$$
\nabla : E \longrightarrow E \otimes V^*
$$
such that $$\nabla (fs) = f \nabla (s) + s\otimes \phi^*(df), $$ where $s$ is any locally
defined holomorphic section of $E$ and $f$ is any locally defined
holomorphic function on $X$. A Lie algebroid connection on $E$ will also be called a $(V, \phi)$--connection on $E$ or simply
a $\phi$--connection on $E$.
\end{definition}

When $(V, \phi)  =  (TX, {\rm Id}_{TX})$, a $(V, \phi)$--connection on $E$ is a 
holomorphic connection on $E$ in the usual sense \cite{At}.

For any Lie algebroid $(V, \phi)$, and any $n \geq 0$, there is a $\mathbb C$--linear homomorphism
$$
d^n_V : \bigwedge\nolimits^n V^*  \longrightarrow \bigwedge\nolimits^{n+1} V^*
$$
which is constructed as follows:
\begin{equation}\label{d0}
d_V^0(f)  = \phi^*(df)
\end{equation}
for any locally defined holomorphic function $f$
on $X$. To construct $d^n_V$ for $n\, \geq\, 1$, take locally defined holomorphic sections $\omega  \in
\bigwedge^n V^*$ and $v_1, \ldots, v_{n+1} \in V$; then
\begin{equation}\label{d10}
d^n_V(\omega)(v_1, \ldots, v_{n+1}) = \sum_{i=1}^n (-1)^{i+1}\phi(v_i)(\omega(v_1, \ldots,
\widehat{v}_i, \ldots,  v_{n+1}))
\end{equation}
$$
+\sum_{1\le i<j\le n+1}(-1)^{i+j} \omega\left ([v_i, v_j], v_1, \ldots,
\widehat{v}_i, \ldots,\,\widehat{v}_j, \ldots, v_{n+1}\right).
$$
It can be easily checked that $d^{n+1}_V \circ d^{n}_V = 0$ for every $n \geq 0$, and it follows that 
$$\left (\bigwedge\nolimits^\bullet V^*, d_V \right) \ =\
\bigoplus_{n\geq 0} \left(\bigwedge\nolimits^n V^*, d^n_V\right)$$
is a differential graded complex; it
is called the Chevalley-Eilenberg-de Rham complex for $(V, \phi)$ (see \cite{BMRT}, \cite{LSX},
\cite{BR} for details).
Note that when $(V, \phi)  =  (TX, {\rm Id}_{TX})$, then $(\bigwedge\nolimits^\bullet V^*, d_V)$
is the holomorphic de Rham complex of $X$.

Given a Lie algebroid connection $\nabla  :  E  \longrightarrow  E\otimes V^*$, consider the
following composition of operators
$$
E   \stackrel{\nabla}{\longrightarrow}  E\otimes V^*  \xrightarrow{\,\nabla \wedge {\rm Id}_{V^*}+
{\rm Id}_E\otimes d^1_V }   E\otimes \bigwedge\nolimits^2 V^*.
$$
It is straightforward to check that this operator $E \longrightarrow E\otimes \bigwedge\nolimits^2 V^*$
is ${\mathcal O}_X$--linear, and hence it is given by a holomorphic section
$$
{\mathcal K}(\nabla) \in  H^0(X,\, \text{End}(E)\otimes \bigwedge\nolimits^2 V^*).
$$
This section ${\mathcal K}(\nabla)$ is called the \textit{curvature} of $\nabla$. The 
Lie algebroid connection $\nabla$ is called \textit{integrable} (or \textit{flat}) if
we have ${\mathcal K}(\nabla)\,=\, 0$.

Observe that on a smooth complex projective variety $X$ over $\mathbb{C}$, the flatness of a 
$(V, \phi)$-connection is not guaranteed when $\mathrm{rank}(V) \ge  2$, even in the case where $\dim X  =  1$ 
or when the Lie algebroid $(V, \phi)$ admits a splitting.

Consider the following direct sum of $\mathbb{C}$-modules:
\begin{equation}\label{eq:jet}
J^1_{V}(E) :=  E\oplus(E\otimes V^*).
 \end{equation}
We can equip $J^1_{V}(E)$ with two $\mathcal{O}_X$-module structures. The first one is coordinate-wise
multiplication $$ f \cdot (s,  \sigma)  :=  (f s,   f \sigma),$$
and the second one is given by
\begin{equation}
\label{eq:mult}
f \cdot (s,  \sigma) := (f s,  f \sigma + s \otimes d_{V}f),
\end{equation}
where $s$ (respectively, $\sigma$) is a locally defined holomorphic section of $E$
(respectively, $E\otimes V^*$) and $f$ is a locally defined holomorphic function on $X$. The $\mathbb{C}$-module
$J^1_{V}(E)$ equipped with the $\mathcal{O}_X$-module structure in \eqref{eq:mult} is called the {\it first order $V$-jet bundle}
associated to $E$. This first order $V$-jet bundle $J^1_{V}(E)$ fits into
the following short exact sequence of $\mathcal{O}_X$-modules,
\begin{equation}\label{eq:2}
0 \longrightarrow  E\otimes V^* \xrightarrow{\, \iota_E\,} J^1_{V}(E)  \xrightarrow{\,p_E \,} E \longrightarrow 0,
\end{equation}
where $\iota_E$ and $p_E$ are the canonical inclusion and projection respectively. 
In general, the short exact sequence in \eqref{eq:2} may not split holomorphically as $\mathcal{O}_X$-modules.
It is easy to see that a holomorphic $(V,\, \phi)$--connection on $E$ is a holomorphic splitting of \eqref{eq:2}.

Let 
\begin{equation}\label{eq:at}
\mathrm{at}_V(E) \in  \mathrm{H}^1(X, \mathrm{End}(E) \otimes V^* )
\end{equation}
denote the extension class of the short exact sequence in \eqref{eq:2}, called the $V$-Atiyah class.
It was noted before that giving a holomorphic splitting of \eqref{eq:2} is equivalent to giving a
holomorphic $(V, \phi)$-connection on $E$. Therefore, 
$E$ admits a $(V, \phi)$-connection if and only if we have $\mathrm{at}_V(E)  =  0$.

The $V$-Atiyah class $\mathrm{at}_V(E)$ in \eqref{eq:at} has functorial property as follows. Let
$$f  :  Y  \longrightarrow  X$$ be a morphism between smooth projective varieties over $\mathbb{C}$. We have its
differential $$df :  TY  \longrightarrow  f^*TX$$. Given a Lie
algebroid $(V, \phi)$ over $X$, consider the following commutative diagram
\begin{equation}\label{eq:cd1}
\begin{gathered}
\xymatrix{
W \ar[d]^{\widehat{\phi}} \ar[r]^{p} & f^*V \ar[d]^{f^*\phi} \\
TY \ar[r]^{df} & f^*TX
}
\end{gathered}
\end{equation} 
where $W  =  TY \times_{f^*TX} f^*V$ is the fiber product, while $\widehat{\phi} : W  \longrightarrow
 TY$ and $p  :  W
 \longrightarrow  f^*V$ are the natural projections from the fiber product. Note that $W$ is the subbundle
of $TY\oplus f^*V$ whose fiber over any $y  \in  Y$ is the subspace of $T_yY\oplus (f^*V)_y$ consisting of all
$(a, b)$, where $a  \in  T_yY$ and $b  \in  (f^*V)_y$, such that $f^*\phi (b)  =  df (a)$.
The $\mathbb C$--Lie algebra structure on $V$ pulls back to a $\mathbb C$--Lie algebra on $f^*V$;
the $\mathbb C$--Lie algebra structures on $TY$ and $f^*V$ together produce a $\mathbb C$--Lie
algebra structure on $W$. Thus $(W,  \widehat{\phi})$ is a Lie algebroid over $Y$. Note that the maps
$p$ and $\widehat\phi$ in \eqref{eq:cd1} are Lie algebra structure preserving.

Now, given a $(V,  \phi)$-connection $\nabla$ on $E  \longrightarrow   X$, the pullback $f^*E 
\longrightarrow  Y$ is equipped with a $(W,  \widehat{\phi})$--connection; this $(W, 
\widehat{\phi} )$--connection is given by the following composition of maps
$$
f^*E   \xrightarrow{\, f^*\nabla \,}  f^*(E \otimes V^*)    =   (f^*E)\otimes (f^*V^*) 
\xrightarrow{\,{\rm Id}_{f^*E}\otimes p^*\,} (f^*E)\otimes W^*,
$$
where $p^*$ is dual of the map $p$ in \eqref{eq:cd1}. The above $(W,  \widehat{\phi})$--connection on $f^*E$ will
be denoted by $f^*\nabla$. Next, the map $f  :  Y  \longrightarrow  X$ induces a natural map on the cohomology groups
$$f^\sharp  :  \mathrm{H}^1 (X,  \mathrm{End}(E) \otimes V^*)  \longrightarrow  \mathrm{H}^1 (Y, 
\mathrm{End}(f^*E) \otimes W^*).$$ It is straightforward to check that
\begin{equation}\label{cc}
f^\sharp (\mathrm{at}_V(E))   =  \mathrm{at}_W(f^*E)
\end{equation}
(see \eqref{eq:at}). Indeed, \eqref{cc} follows immediately from the fact that we have a commutative diagram
$$
\xymatrix@R=1.5cm@C=1.5cm{
0 \ar[r] 
& (f^*E)\otimes (f^*V^*) \ar[r]^{f^*\iota_E} \ar[d]_{\mathrm{Id}_{f^*E}\otimes p^*}
& f^*J^1_{V}(E) \ar[r]^{f^*p_E} \ar[d]^{\beta}
& f^*E \ar[r] \ar@{=}[d]
& 0 \\
0 \ar[r]
& (f^*E)\otimes W^* \ar[r]
& J^1_{W}(f^*E) \ar[r]
& f^*E \ar[r]
& 0
}
$$
(see \eqref{eq:2}), where $\beta$ is the restriction of the homomorphism
$$
{\rm Id}_{f^*E}\oplus ({\rm Id}_{f^*E}\otimes p^*)   :  
f^* (E\oplus(E\otimes V^*))  =   (f^* E)\oplus (f^* (E\otimes V^*))  \longrightarrow
(f^* E)\oplus ((f^*E)\otimes W^*).
$$

\section{Criterion for existence of Lie algebroid connection}

Let $X$ be embedded in the projective space $\mathbb{CP}^N$. Given a general hypersurface $H_n$ of degree $n$ 
in $\mathbb{CP}^N$, let $X_n  :=  X \cap H_n $ denote the hypersurface of degree $n$ in $X$. Given a
holomorphic vector bundle $E$ over $X$, denote by $E\big\vert_n$ the restriction
$E\big\vert_{X_n}$ of $E$ to $X_n$.

Consider the following short exact sequence of coherent sheaves over $X_n$
\begin{equation}
\label{eq:s1}
0 \longrightarrow  TX_n  \xrightarrow{\,\iota \,} TX\big\vert_{n}  \xrightarrow{\, \tau_n
\, }\ \mathcal{O}_X(X_n)\big\vert_{X_n}  =  N_{X_n}  \longrightarrow   0,
\end{equation}
where $N_{X_n}$ is the normal bundle of $X_n$, and $\tau_n$ is the projection to the normal bundle;
the above isomorphism $\mathcal{O}_X(X_n)\big\vert_{X_n}  =  N_{X_n}$ is given by the Poincar\'e adjunction
formula. Given a Lie algebroid $(V, \phi)$ over $X$, let
\begin{equation}
\label{eq:s2}
\psi  :  V\big\vert_n   \longrightarrow  \mathcal{O}_X(X_n)\big\vert_{X_n}
\end{equation}
be the following composition of homomorphisms
$$
V\big\vert_n   \xrightarrow{\, \phi\big\vert_{X_n} \,}\ TX\big\vert_n  \xrightarrow{
\tau_n } \mathcal{O}_X(X_n)\big\vert_n.
$$
Denote 
\begin{equation}
\label{eq:s3}
V_n := \mathrm{Ker}(\psi)  =  (\phi\big\vert_{X_n})^{-1} (TX_n)  \subset V\big\vert_{n},
\end{equation}
where $\psi$ is the homomorphism in \eqref{eq:s2}; the restriction of $\phi\big\vert_{X_n}$ to
$V_n$ will be denoted by $\phi_n$. Then the pair 
\begin{equation}
\label{eq:V_n}
(V_n, \phi_n)
\end{equation}
is a Lie algebroid over $X_n$. Moreover, we get a short exact sequence of $\mathcal{O}_{X_n}$-modules
\begin{equation}
\label{eq:s4}
0  \longrightarrow  V_n  \xrightarrow{j_n}  V\big\vert_{n} \xrightarrow{
\psi } \mathrm{Im}(\psi) \longrightarrow  0
\end{equation}

If $\psi$ in \eqref{eq:s2} is surjective --- this happens, for example, if $\phi$ is surjective --- then 
$$
\mathrm{Im}(\psi) \cong  \mathcal{O}_X(X_n)\big\vert_{X_n} =  N_{X_n}.
$$
Consequently, when $\phi$ is surjective, from \eqref{eq:s4} we get the following short exact sequence
\begin{equation}\label{eq:s5}
0  \longrightarrow V_n  \xrightarrow{j_n}  V\big\vert_{n} \xrightarrow{\,
\psi\, }\ \mathcal{O}_X(X_n)\big\vert_{n} \longrightarrow  0.
\end{equation}

\begin{theorem}\label{thm:main}
Assume that $\mathrm{dim}_{\mathbb{C}} (X)  \geq  3$ and the anchor map $\phi  :  V  \longrightarrow 
TX$ is surjective. A holomorphic vector bundle $E$ over $X$ admits a holomorphic $(V,\,\phi)$-connection if
and only if the restriction $E\big\vert_{X_n}$ admits a holomorphic $(V_n, \phi_n)$-connection
for sufficiently large $n$, where $(V_n, \phi_n)$ is the Lie algebroid constructed in \eqref{eq:V_n}.
\end{theorem}

\begin{proof} 
Let $\mathcal{F}$ be a coherent sheaf of $\mathcal{O}_X$-modules on $X$. Then, we have the exact
sequence of coherent sheaves
\begin{equation} \label{eq:s6}
0 \longrightarrow \mathcal{F} \otimes \mathcal{O}_X(-X_n) \longrightarrow
\mathcal{F} \longrightarrow i_*\left( \mathcal{F}\big\vert_{X_n} \right) \longrightarrow 0,
\end{equation}
where $ i :  X_n  \hookrightarrow   X$ is the inclusion map. Let $\mathcal{O}_X(1)
:= i^* \mathcal{O}_{\mathbb{CP}^N} (1)$. Then, $\mathcal{O}_X(1)$ is very ample and $\mathcal{O}_X(X_n)
 \cong  \mathcal{O}_X(n)$. Since
$\dim X \geq  3$, it follows from the Serre's duality theorem \cite[Theorem 7.6, p.~243]{Ha} and Serre's theorem
\cite[Theorem 5.2, p.~228]{Ha} that for $n  \gg  0$, we have
\begin{equation} \label{eq:s7}
\mathrm{H}^i\left(X,  \mathcal{F} \otimes \mathcal{O}_X(-X_n) \right)  =  0, \quad \forall\,  i  = 1,  2.
\end{equation}
Consider the long exact sequence of cohomologies associated with the short exact sequence
\eqref{eq:s6}
$$
\mathrm{H}^1\left(X,\, \mathcal{F} \otimes \mathcal{O}_X(-X_n) \right) 
 \longrightarrow  \mathrm{H}^1(X,  \mathcal{F})  \longrightarrow
\mathrm{H}^1\left(X,  i_*\left( \mathcal{F}\big\vert_{X_n} \right)\right)  = 
\mathrm{H}^1\left(X_n, \mathcal{F}\big\vert_{X_n}\right)\,$$
$$
\longrightarrow 
\mathrm{H}^2\left(X, \mathcal{F} \otimes \mathcal{O}_X(-X_n) \right).
$$
Using \eqref{eq:s7} in this exact sequence we get an isomorphism
\begin{equation} \label{eq:s8}
\mathrm{H}^1(X, \mathcal{F})  \xrightarrow{\, \cong \,}
\mathrm{H}^1\left(X_n, \mathcal{F}\big\vert_{X_n}\right).
\end{equation}
Since $\phi$ is surjective (by assumption), we have the short exact sequence in \eqref{eq:s5}. Consider
its dual exact sequence
\begin{equation}
\label{eq:s9}
0  \longrightarrow  \mathcal{O}_X(-X_n)\big\vert_n  \longrightarrow  V^*\big\vert_n
\longrightarrow  V^*_n  \longrightarrow   0. 
\end{equation}
Tensoring the short exact sequence in \eqref{eq:s9} with $\mathrm{End} (E)\big\vert_{X_n}$ the following
short exact sequence is obtained:
\begin{equation}
\label{eq:s10}
0 \longrightarrow  \left( \mathrm{End} (E) \otimes \mathcal{O}_X(-X_n) \right)\big\vert_n
 \longrightarrow  \left( \mathrm{End} (E) \otimes V^* \right)\big\vert_n  \longrightarrow  
\mathrm{End} (E)\big\vert_n \otimes V^*_n  \longrightarrow  0.
\end{equation}

Setting $\mathcal{F}  =  \mathrm{End} (E) \otimes \mathcal{O}_X(-X_n) $ in the short exact sequence in
\eqref{eq:s6}, we get the following
\begin{equation}\label{eq:s11}
0  \longrightarrow  \mathrm{End} (E)\otimes \mathcal{O}_X(-2 X_n)   \longrightarrow  \mathrm{End} (E)\otimes
\mathcal{O}_X(-X_n)   \longrightarrow  \left( \mathrm{End} (E) \otimes \mathcal{O}_X(-X_n) \right)\big\vert_n
\longrightarrow  0.
\end{equation}
Consider the long exact sequence of cohomologies associated to \eqref{eq:s11}:
\begin{equation*}
\longrightarrow  \mathrm{H}^1\left(X,  \mathrm{End} (E)\otimes \mathcal{O}_X(- X_n) \right)   \longrightarrow 
\mathrm{H}^1\left(X_n,  \left( \mathrm{End} (E) \otimes \mathcal{O}_X(-X_n) \right)\big\vert_n \right)
 \end{equation*}
\begin{equation}\label{eq:s11b}
\longrightarrow \mathrm{H}^2\left(X,  \mathrm{End} (E)\otimes \mathcal{O}_X(-2 X_n) \right) 
\longrightarrow.
\end{equation}
Note that from \eqref{eq:s7} we have 
\begin{equation}
\label{eq:s12}
\mathrm{H}^2\left(X,  \mathrm{End} (E)\otimes \mathcal{O}_X(-2 X_n) \right)  =  0  =  
\mathrm{H}^1\left(X,  \mathrm{End} (E)\otimes \mathcal{O}_X(- X_n) \right)
\end{equation}
for all $n  \gg  0$. This and \eqref{eq:s11b} combine together to give that
\begin{equation}
\label{eq:s13}
\mathrm{H}^1\left(X_n,  \left( \mathrm{End} (E) \otimes \mathcal{O}_X(-X_n) \right)\big\vert_n \right)  =  0
\end{equation}
for all $n  \gg   0$.

Consider the long exact sequence of cohomologies for the short exact sequence in \eqref{eq:s10}:
$$
\longrightarrow  \mathrm{H}^1\left(X_n, \left( \mathrm{End} (E) \otimes
\mathcal{O}_X(-X_n) \right)\big\vert_n \right)
 \longrightarrow  \mathrm{H}^1 (X_n, \left( \mathrm{End} (E)\otimes V^* \right)\big\vert_n) 
\longrightarrow  \mathrm{H}^1 (X_n, \mathrm{End} (E)\big\vert_n \otimes V^*_n)  \longrightarrow .
$$
In view of \eqref{eq:s13}, it gives an injective homomorphism
\begin{equation}
\label{eq:s14}
\mathrm{H}^1 (X_n,  \left( \mathrm{End} (E)\otimes V^* \right)\big\vert_n)  \hookrightarrow 
\mathrm{H}^1 (X_n,  \mathrm{End} (E)\big\vert_n \otimes V^*_n).
\end{equation}
Also, from the isomorphism in \eqref{eq:s8} we have 
\begin{equation}
\label{eq:s15}
\mathrm{H}^1(X, \mathrm{End} (E)\otimes V^*)\ \xrightarrow{\, \cong \,}
\mathrm{H}^1 (X_n, \left( \mathrm{End} (E)\otimes V^* \right)\big\vert_n)
\end{equation}
for all $n  \gg  0$. Now, from \eqref{eq:s14} and \eqref{eq:s15}, the homomorphism 
$$
i^{\sharp}  :  \mathrm{H}^1(X,\, \mathrm{End} (E)\otimes V^*)  \longrightarrow
\mathrm{H}^1 (X_n, \mathrm{End} (E)\big\vert_n \otimes V^*_n)
$$
induced by the natural inclusion map $i  :  X_n  \hookrightarrow  X$ is injective, so
\begin{equation}\label{eq:s16}
i^{\sharp}  :  \mathrm{H}^1(X,  \mathrm{End} (E)\otimes V^*)  \hookrightarrow 
\mathrm{H}^1 (X_n,  \mathrm{End} (E)\big\vert_n \otimes V^*_n).
\end{equation}

Note that $i^{\sharp} (\mathrm{at}_V(E))   =  \mathrm{at}_{V_n}(E\big\vert_n)$ by the functoriality of
the Atiyah class. Thus, from \eqref{eq:s16} it follows immediately that for all $n  \gg  0$, we have $\mathrm{at}_{V_n}(E\big\vert_n)  =   0$ if and only if $\mathrm{at}_V(E)  =  0$.
This completes the proof. 
\end{proof}

Let $X$ be a smooth complex projective variety of dimension $d   \geq  1$. 
Let $\mathrm{Div}(X)$ be a group of all (Weil) divisors on $X$. Recall that every divisor
$Z   \in  \mathrm{Div}(X)$ can be written uniquely as (finite sum) $Z  =  \sum_{i} a_i Z_i$, where
$a_i  \in  \mathbb{Z}$ and each $Z_i  \subset  X$ is a prime divisor, i.e., an irreducible closed subvariety
of codimension $1$. Let $Z  =  \sum_{i} a_i Z_i$ and 
$Z' = \sum_{j} b_j Z'_j$ be two divisors on $X$. We say that $Z$ and $Z'$ \emph{meet properly} if
for every pair of prime divisors $Z_i$ and $Z'_j$ appearing with non-zero coefficients $a_i  \neq  0$
and $b_j  \neq  0 $ in $Z$ and $Z'$, respectively, one has
$$
\dim ( Z_i \cap Z'_j )   =  d - 2 .
$$

Now assume that $Z$ and $Z'$ are \emph{reduced effective divisors}. Then each of them is a union of distinct 
prime divisors with coefficient $1$. Under this assumption, if $Z$ and $Z'$ meet properly, the intersection 
$Z \cap Z'$ has pure dimension $d-2$, and hence $Z \cap Z'$ defines a Weil divisor on both $Z$ and $Z'$.

Let $Z$ be a reduced effective divisor on $X$. Then, consider the following Lie algebroid
\begin{equation}
\label{eq:red}
j :  W  =  TX \otimes \mathcal{O}_X(-Z)  \hookrightarrow  TX
\end{equation}
 with (non-surjective) anchor map as inclusion morphism $j$. Let $X_n$ be a smooth
hypersurface of degree $n$ in $X$. Suppose that $Z$ and $X_n$ meet properly, then 
$$
Z_n  :=  Z \cap X_n
$$ is a divisor on $X_n$. We get a Lie algebroid 
\begin{equation}
\label{eq:red-n}
j_n  :  W_n  :=  TX_n \otimes \mathcal{O}_{X_n}(-Z_n)  \hookrightarrow  TX_n
\end{equation}
on $X_n$. Using similar techniques as in Theorem \ref{thm:main}, we can show the following theorem. 

\begin{theorem} \label{thm:red}
Assume that $\dim_{\mathbb{C}}(X)  \ge  3$, and let $Z  \subset  X$ be a reduced effective divisor
on $X$. Consider the holomorphic Lie algebroid $j :  W  =  TX \otimes \mathcal{O}_X(-Z)  \hookrightarrow  TX $
on $X$. Then $E$ admits a $(W, j)$-connection if and only if for some smooth
hypersurface $X_n$ of sufficiently large degree $n$, which intersects $Z$ properly, the vector
bundle $E\big\vert_{X_n}$ on $X_n$ admits a $(W_n,  j_n)$-connection, where $(W_n,  j_n)$ is
defined in \eqref{eq:red-n}.
\end{theorem}

\begin{remark}
Theorem \ref{thm:main} and Theorem \ref{thm:red} provide restriction criteria for the existence of Lie algebroid connections in higher dimensions under natural assumptions on the anchor map. These results extend Atiyah’s classical theorem to a broad class of Lie algebroids (e.g. Atiyah algebroid associated with a vector bundle), including those arising from logarithmic structures. Establishing analogous criteria for arbitrary
Lie algebroids or foliations remains an important open problem.
\end{remark}

\section*{Acknowledgements}

We thank the referee for helpful comments.
I.B. is partially supported by a J. C. Bose Fellowship (JBR/2023/000003).
A.S. is partially supported by ANRF/ARGM/2025/000670/MTR.
A.S. would like to thank Shiv Nadar University for their hospitality during the conference ``An International Conference on Algebraic and Analytic Geometry" from 1st to 5th December 2025, where the work was initiated.

\end{document}